\documentclass[12pt]{amsart}
\usepackage[pagewise]{lineno}
\usepackage{amscd,amssymb}
\usepackage[arrow,matrix]{xy}
\usepackage[urlcolor=blue,linkcolor=red,colorlinks=true]{hyperref}
\usepackage{mathdots}
\usepackage{mathrsfs}
\usepackage{amsthm}
\usepackage{bm}
\usepackage{amscd}
\usepackage{mathtools}

\theoremstyle{plain}
\newtheorem{thm}[subsection]{Theorem}
\newtheorem{lem}[subsection]{Lemma}

\newtheorem{cor}[subsection]{Corollary}

\theoremstyle{definition}

\newcommand{\bb}{\mathbb}

\newcommand{\p}{\partial}

\newcommand{\Res}{\text{Res}}

\newcommand{\rd}{\mathrm{d}}

\newcommand{\spa}{\mathrm{span}}

\newcommand{\Grass}{\mathrm{Grass}}

\newcommand{\Hom}{\mathrm{Hom}}

\newcommand{\Ker}{\mathrm{Ker}}
\newcommand{\A}{\mathbf{A}}

\begin{document}
	\date{}
		
	\title[Deformation of Milnor Algebras]{Deformation of Milnor Algebras}
		
	\author[ZHENJIAN WANG]{ ZHENJIAN WANG  }
	\address{YMSC, Tsinghua University, 100084 Beijing, China}
	\email{wzhj@mail.tsinghua.edu.cn}
		
		
		
	\begin{abstract}
    We investigate deformations of Milnor algebras of smooth homogeneous polynomials, and prove in particular that any smooth degree $d$ homogeneous polynomial in $n+1$ variables that is not of Sebastiani-Thom type is determined by the degree $k$ homogeneous component of its Jacobian ideal for any $d-1\leq k\leq (n+1)(d-2)$. Our results generalize the previous result on the reconstruction of a homogeneous polynomial from its Jacobian ideal.
	\end{abstract}
	\maketitle
%    \tableofcontents

\section{Introduction}

The classical theory of variation of Hodge structures for smooth hypersurfaces in a complex projective space gives a variation of Milnor algebras of homogeneous polynomials. The celebrated generic Torelli theorem for hypersurfaces is almost reduced to the study of injectivity of some mappings concerning the deformation of Milnor algebras, see \cite[Subsection 6.3.2, p.179]{Voi} and \cite{Don}. Nevertheless, the homogeneous components of the Milnor algebra involved there are of specific degrees, namely degrees of the form $pd-n-1$. In this note, we will investigate homogeneous components of all degrees of the Milnor algebra.

To fix notation, let $S=\bb{C}[x_0,\cdots,x_n]$ be the homogeneous coordinate ring of the complex projective space $\bb{P}^n$,
$$
S=\bigoplus_{d=0}^\infty S_d,
$$
where $S_d$ is the vector space of homogeneous polynomials of degree $d$. Given a homogeneous polynomial $f\in S_d$, denote
$$
J(f)=(\p f/\p x_0,\p f/\p x_1,\cdots, \p f/\p x_n),
$$
known as the Jacobian ideal of $f$. Set $M(f)=S/J(f)$, known as the Milnor algebra of $f$. The algebra $M(f)$ has the natural grading
$$
M(f)=\bigoplus_{k=0}^\infty M(f)_k,
$$
where $M(f)_k=S_k/(J(f)\cap S_k)$.

We say that $f\in\bb{P}(S_d)$ is a \emph{smooth} polynomial if the hypersurface $V_f: f=0$ in $\bb{P}^n$ is a smooth hypersurface. The discriminant defines a divisor $\Delta\subset\bb{P}(S_d)$ such that the complement $\bb{P}(S_d)_{\Delta}$ parameterizes smooth homogeneous polynomials of degree $d$.

We say that a polynomial $f\in S_d$ is of \emph{Sebastiani-Thom type (ST type)} or a \emph{direct sum} if $f$ can be represented as
\begin{equation}\label{STtype}
f(x_0,\cdots,x_n)=f_1(x_0,\cdots,x_\ell)+f_2(x_{\ell+1},\cdots,x_n)
\end{equation}
for a choice of homogeneous coordinates $\{x_i\}_{i=0}^n$ of $\bb{P}^n$ and some $0\leq\ell<n$; see \cite{UY,ZW,BBKT}. For various characterizations of polynomials of ST type, we refer to \cite{Fed2}.
Denote by $\mathscr{U}\subset\bb{P}(S_d)_{\Delta}$ the set of all smooth homogeneous polynomials that are \emph{not} of ST type.

It is well-known that $\dim M(f)_k$ does not depend on the concrete equation of $f$ for smooth $f$'s (see for instance \cite[Proposition 7.22, p.108]{Dim87}); we denote this dimension by $a_{n,d}(k)$. Let $\Grass(S_k,a_{n,d}(k))$ be the Grassmannian parameterizing all $a_{n,d}(k)$ dimensional \emph{quotient spaces} of $S_k$, then we have the following map
\begin{equation}\label{varphik}
\varphi_k:\,\bb{P}(S_d)_{\Delta}\to\Grass(S_k,a_{n,d}(k)),
\end{equation}
defined by $f\mapsto M(f)_k$.

More generally, denote $\Grass(n+1, S_{d-1})$ the Grassmannian of \emph{linear subspaces} of dimension $n+1$ of the space of degree $d-1$ homogeneous polynomials $S_{d-1}$. Following \cite[Subsection 1.2]{Fed1},
given $W\in\Grass(n+1,S_{d-1})$, we form the ideal $I_W:=(W)$ and the quotient algebra $M(W)=S/I_W$. Let $g_0,\cdots, g_n$ be a basis of $W$, then the sequence $g_0,\cdots,g_n$ is a regular sequence if and only if $I_W$ is a complete intersection ideal if and only if $M(W)$ is a standard local Artinian
Gorestein algebra of socle degree $T:=(n+1)(d-2)$ if and only if the resultant of $g_0,\cdots,g_n$ is nonzero; we refer to \cite[Chapter 13]{GKZ} for the definition and basic properties of the resultant. Therefore, there exists a divisor $\Res\subset\Grass(n+1,S_{d-1})$ parameterizing all $W$ such that $I_W$ is not a complete intersection ideal. We denote by $\Grass(n+1,S_{d-1})_{\Res}$ the affine complement of $\Res$. For more discussions about ideals of the form $I_W$, see \cite[Subsection 1.2]{Fed1} .

For $W\in\Grass(n+1,S_{d-1})_{\Res}$, we have $\dim M(W)_k=a_{n,d}(k)$ by \cite[Proposition 7.22, p.108]{Dim87}. Hence the assignment $W\mapsto M(W)_k$ defines a map
\begin{equation}\label{Phik}
\Phi_k:\,\Grass(n+1,S_{d-1})_{\Res}\to\Grass(S_k, a_{n,d}(k)).
\end{equation}
Our first result is the following theorem.

\begin{thm}\label{injective1}
For any $d-1\leq k\leq T=(n+1)(d-2)$, the map $\Phi_k$ is an immersion, that is, it is injective and the differential $\rd\Phi_k$ is also injective at any point of $\Grass(n+1,S_{d-1})_{\Res}$.
\end{thm}

Using our pervious result on determination of a polynomial by its Jacobian ideal (see \cite[Theorem 1.1]{ZW} and \cite[Lemma 3]{UY}), we further prove the following result.
\begin{thm}\label{injective2}
For $d-1\leq k\leq T=(n+1)(d-2)$, the restriction of the map $\varphi_k$ (defined in \eqref{varphik}) to $\mathscr{U}$,
$$
\varphi_k:\,\mathscr{U}\to\Grass(S_k,a_{n,d}(k)),
$$
is an immersion.
\end{thm}

In particular, we have that a smooth homogeneous polynomial $f\in\mathscr{U}$ can be reconstructed from the degree $k$ homogeneous component of its Jacobian ideal $J(f)$ for any $k$ satisfying $d-1\leq k\leq T=(n+1)(d-2)$. This gives a generalization of the previous results, in the case of smooth polynomials, in \cite{ZW} or \cite{UY}.

We will also investigate the map $\varphi_k$ defined in \eqref{varphik} and discuss its fibers over $\varphi_k(f)$ for homogeneous polynomials $f$'s that are of ST type, see Section \ref{secST}.

Our results are related to the problem of characterizing the hypersurface singularity $\widehat{V}_f=\{x\in\bb{C}^{n+1}:f(x)=0\}$ at the origin $0$ of $\bb{C}^{n+1}$
using the Milnor algebra $M(f)$. In fact, the characterization problem of a singularity by its algebraic data can be proposed and solved in a much more general setting, see \cite{GH} and the references therein. As a general philosophy in singularity theory, the Milnor algebra $M(f)$ is closely connected to the topology and geometry of the hypersurface singularity $(\widehat{V}_f,0)\subset(\bb{C}^{n+1},0)$. Instead of giving characterizations of a singularity by \emph{algebras} or \emph{modules} derived from it, as in \cite{GH}, here using Theorem \ref{injective2}, we can give a characterization of the \emph{isolated} hypersurface singularity $(\widehat{V}_f,0)$ just by a single homogeneous component $M(f)_k$ of the Minor algebra $M(f)$ for any $d-1\leq k\leq (n+1)(d-2)$ with $d=\deg f$. This conclusion can obviously be extended to an isolated complete intersection singularity by using Theorem \ref{injective1}.

Of course, our results concern only the case when the hypersurface $(\widehat{V}_f,0)$ is an isolated singularity. It is natural to extend these results to the case where the singularities of the hypersurface $\widehat{V}_f$ have positive dimension. However, the tools used in this note cannot be directly applied in the extended case because they depend heavily on the condition that $M(f)$ is a local Artinian Gorestein algebra which holds only when 0 is an isolated singularity of $\widehat{V}_f$. In addition, heuristically, the results in \cite{GH} also show that any possible extension must be more complicated and more technical than our results above; see also the results in \cite{ZW} concerning the non-smooth homogeneous polynomials.

We hope the results in this note can be applied to the study of Lefschetz properties for Milnor algebras. In fact, this is an important impetus to our present work. As it is well-known, the strong Lefschetz property holds for $M(f)$ for a generic $f$. Our na\"ive idea is to investigate the Lefschetz properties by deforming the Milnor algebras. For an excellent exposition for the Lefschetz properties, we refer to \cite{HMNW,MN,HMMNW}. In addition, the strong Lefschetz property for $M(f)$ where $f$ is of ST type can be reduced to that where $f$ is not of ST type (see \cite[Theorem 3.10]{HW} and \cite[Proposition 3.77, p.137]{HMMNW}), since $M(f)$ is the tensor product of $M(f_1)$ and $M(f_2)$ when $f$ is represented as in \eqref{STtype}.  This is an important reason why we specifically investigate the set $\mathscr{U}$ in this paper; another reason is about the determination of a homogeneous polynomial by its Jacobian ideal, see the proof of Corollary \ref{cor: equi} below.

We would like to thank Professor Herwig Hauser for the reference \cite{GH}. We thank an anonymous referee for useful remarks and suggestions. We also thank Yau Mathematical Sciences Center for financial support and stimulating working atmosphere.

\section{Polar paring and Macaulay inverse systems}

\subsection{Polar pairing}
Let $S=\bb{C}[x_0,\cdots,x_n]$ and $R=\bb{C}[z_0,\cdots,z_n]$ be two polynomial rings. There is a natural action of $S$ on $R$ by the ``polar paring"
$$
S\times R\to R
$$
defined by
$$
\bigl(f(x_0,\cdots,x_n),Q(z_0,\cdots,z_n)\bigr)\mapsto f\cdot Q:=f(\p/\p z_0,\cdots,z_n)Q(z_0,\cdots,z_n).
$$
It induces perfect parings $S_\rho\times R_\rho\to\bb{C}$ for every $\rho\in\bb{N}$.
In particular, for $f\in S_\rho$ written as
$$
f(x_1,\cdots,x_n)=\sum\limits_{|\alpha|=\rho}a_\alpha x^\alpha
$$
and $Q\in R_\rho$ written as
$$
Q(z_0,\cdots,z_n)=\sum\limits_{|\alpha|=\rho}b_\alpha z^\alpha
$$
we have
$$
f\cdot Q=\sum_{|\alpha|=\rho}\alpha!a_\alpha b_\alpha .
$$
Define the polynomial $q=\sum\limits_{|\alpha|=\rho}b_\alpha x^\alpha$, or equivalently $q(x_0,\cdots,x_n)=Q(x_0,\cdots,x_n)$, and define the inner product of $f$ and $q$ by
\begin{equation}\label{innerproduct}
\langle f,q\rangle=\sum_{|\alpha|=\rho}\alpha! a_\alpha b_\alpha=f\cdot Q.
\end{equation}

For any linear space $E\subset S_\rho$, with respect to the above inner product $\langle\ ,\ \rangle$, we have its orthogonal complement, denoted by $E^{\perp}$.

\subsection{Macaulay inverse system}
Let $I\subset S$ be a Gorestein ideal and $\nu$ the socle degree of the algebra $\mathcal{A}=S/I$. Recall that a (homogeneous) Macaulay inverse system of $\mathcal{A}$ is an element $Q_\mathcal{A}\in\bb{P}(R_\nu)$ such that $I$ is equal to the apolar ideal $Q_\mathcal{A}^\perp$, namely,
$$
I=\{f\in S\,:\, f\cdot Q_\mathcal{A}=0\}
$$
(see \cite[Lemma 2.12]{IK} or \cite[Exercise 2.17]{Eis}).

Let $W=\spa\langle g_0,\cdots,g_n\rangle$ such that $W\in\Grass(n+1,S_{d-1})_{\Res}$, the \emph{associated form} ${\bf A}_W:={\bf A}(g_0,\cdots,g_n)\in\bb{P}(R_{T})$ (recall that $T=(n+1)(d-2)$) gives the Macaulay inverse system for $S_W=S/I_W$; see \cite[Proposition 2.1]{AI1}. We write
$$
{\bf A}_W=\sum_{|\alpha|=T}c_\alpha z^\alpha.
$$
In this case, define ${\bf B}_W\in\bb{P}(S_{T})$ by
$$
{\bf B}_W=\sum_{|\alpha|=T}c_\alpha x^\alpha.
$$
The polynomial ${\bf B}_W$, by definition, determines and is determined by ${\bf A}_W$.
Moreover, by the definition of Macaulay inverse systems, we have that $(I_W)_{T}^\perp=\bb{C}{\bf B}_W$, namely, the line $\bb{C}{\bf B}_W$ is exactly the orthogonal complement of $(I_W)_{T}$ with respect to the inner product $\langle\ ,\ \rangle$ on $S_{T}$. Therefore, $\A_W\in\bb{P}(R_{T})$ is uniquely determined by $(I_W)_{T}$.

\begin{lem}\label{Macaulayinv}
For two points $U,W\in\Grass(n+1,S_{d-1})_{\Res}$, the following statements are equivalent:

(1) $U=W$;

(2) $I_U=I_W$;

(3) For any $k$ satisfying $d-1\leq k\leq T=(n+1)(d-2)$, we have $(I_U)_k=(I_W)_k$;

(4) For some $k$ satisfying $d-1\leq k\leq T$, we have $(I_U)_k=(I_W)_k$;

(5) $(I_U)_{T}=(I_W)_{T}$.
\end{lem}

\begin{proof}
It is obvious that (1), (2), (3) are all equivalent and (3) implies (4).

(4)$\Rightarrow$(5): Since $I_U$ is generated by polynomials all of which have degree $d-1$, we have that $(I_U)_{T}$ is the image of $S_{T-k}\times(I_U)_k$ under the multiplication map
$S_{T-k}\times S_k\to S_{T}$. Hence $(I_U)_{T}=(I_W)_{T}$ whenever $(I_U)_k=(I_V)_k$ for $d-1\leq k\leq T$.

(5)$\Rightarrow$(1): This is clear once we note that $\A_U$ can be uniquely determined by $(I_U)_{T}$, and $I_U$ is the apolar ideal $\A_U^\perp$.
\end{proof}

Recall that as it is shown in the introduction, $a_{n,d}(k)=\dim M(f)_k$ for any $f\in\bb{P}(S_d)_{\Delta}$, which is also the dimension of $S_k/(I_W)_k$ for any $W\in\Grass(n+1,S_{d-1})_{\Res}$. Denote $b_{n,d}(k)=\dim S_k-a_{n,d}(k)$ and let $\Grass(b_{n,d}(k), S_k)$ be the Grassmannian parameterizing all $b_{n,d}(k)$ dimensional linear \emph{subspaces} of $S_k$. For a subspace $E\subset S_k$ of dimension $b_{n,d}(k)$, we obtain the quotient space $S_k/E$ of dimension $a_{n,d}(k)$; and the mapping $S_k\supset E\mapsto S_k/E$ clearly defines an isomorphism between the Grassmannians $\Grass(b_{n,d}(k), S_k)$ and $\Grass(S_k, a_{n,d}(k))$. Then to prove Theorem \ref{injective1}, it suffices to prove the following theorem.

\begin{thm}\label{injective}
For any $d-1\leq k\leq T$, the assignment $W\mapsto(I_W)_k$ defines an immersion
\begin{equation}\label{varphim}
\Psi_k:\,\Grass(n+1,S_{d-1})_{\Res}\to\Grass(b_{n,d}(k), S_k),
\end{equation}
that is $\Psi_k$ is injective and the differential $\rd\Psi_k$ is also injective at any point of $\Grass(n+1,S_{d-1})_{\Res}$.
\end{thm}

\begin{proof}
The injectivity of $\Psi_k$ follows from the equivalence (1)$\Leftrightarrow$(4) in Lemma \ref{Macaulayinv}.

Given $W\in\Grass(n+1,S_{d-1})_{\Res}$ such that $W=\spa\langle g_0,\cdots,g_n\rangle$. For any $h\in T_{W}(\Grass(n+1,S_{d-1})_{\Res})\simeq\Hom(W,S_{d-1}/W)$, choose $h_i\in S_{d-1}$ such that $h(g_i)=h_i\mod W$ for $i=0,\cdots, n$. Then if $h\in\Ker(\rd\Psi_k)_{W}$, we have $(\rd\Psi_k)_{W}(h)=0$ as an element in $\Hom((I_W)_k,S_k/(I_W)_k)$.
A direct computation gives that
$$
(\rd\Psi_k)_{W}(h)((I_W)_k)=(I_H)_k+(I_W)_k\mod(I_W)_k,
$$
where $H=\spa\langle h_0,\cdots, h_n\rangle$. It follows from $(\rd\Psi_k)_{W}(h)=0$ that $(I_H)_k\subset (I_W)_k$.

For $t\in\bb{C}^*$ and $|t|$ sufficiently small, we have that $W_t:=\spa\langle g_0+th_0,\cdots, g_n+h_n\rangle$ satisfies $W_t\in\Grass(n+1,S_{d-1})_{\Res}$. It then follows from $(I_H)_k\subset (I_W)_k$ that
$(I_{W_t})_k\subset(I_W)_k$, hence $(I_{W_t})_k=(I_W)_k$ because $\dim(I_{W_t})_k=b_{n,d}(k)=\dim(I_W)_k$. Therefore $W_t=W$ by (4)$\Rightarrow$(1) in Lemma \ref{Macaulayinv}. It follows that $h_i\in W$ for $i=0,\cdots, n$ and thus $h=0$ as an element of $T_{W}(\Grass(n+1,S_{d-1})_{\Res})$.

Since $h$ can be arbitrarily chosen, $(\rd\Psi_k)_{W}$ is injective. We are done.
\end{proof}

\section{Variation of Milnor algebras}

\subsection{Polynomials not of ST type}
Recall that $\mathscr{U}\subset\bb{P}(S_d)$ denotes the space of smooth homogeneous polynomials of degree $d$ that are not of ST type, or equivalently, the space of smooth hypersurfaces whose defining equations are not of ST type. From the proof of \cite[Corollary 6.1]{ZW}, we have that $\mathscr{U}$ is a Zariski open subset of $\bb{P}(S_d)_\Delta$.

For $f\in\mathscr{U}$, recall that
$J(f)$ denotes the Jacobian ideal of $f$ and $M(f)=S/J(f)$ the Milnor algebra. For $k\geq d-1$, we denote by $E_k(f)=J(f)\cap S_k$.
Then $\dim E_k(f)=b_{n,d}(k)=\dim S_k-a_{n,d}(k)$ is independent of $f\in\mathscr{U}$. Moreover, since $\p f/\p x_0,\cdots,\p f/\p x_n$ form a regular sequence and $J(f)=I_{E_{d-1}(f)}$, from Lemma \ref{Macaulayinv}, we immediately get the following corollary.

\begin{cor}\label{cor: equi}
Given $f,g\in\bb{P}(S_d)$ and $f\in\mathscr{U}$, the following conditions are equivalent:

(1) $E_{d-1}(f)=E_{d-1}(g)$;

(2) $J(f)=J(g)$;

(3) For any $k$ satisfying $d-1\leq k\leq T=(n+1)(d-2)$, we have $E_k(f)=E_k(g)$;

(4) For some $k$ satisfying $d-1\leq k\leq T$, we have $E_k(f)=E_k(g)$;

(5) $E_T(f)=E_T(g)$;

(6) $f=g$.
\end{cor}

\begin{proof}
The equivalences among the first five statements follow from Lemma \ref{Macaulayinv}; we here just note that any one of these conditions imply that $E_{T+1}(g)=S_{T+1}$, hence $g$ is also smooth and thus $J(g)$ is a complete intersection ideal.

The equivalence (1)$\Leftrightarrow$(6) follows from \cite[Theorem 1.1]{ZW} or \cite[Lemma 3]{UY}.
\end{proof}

Now we are ready to prove Theorem \ref{injective2}. Similar to the proof of Theorem \ref{injective1}, it is sufficient to prove the following theorem.

\begin{thm}\label{mainthm}
For any $d-1\leq k\leq T$, the assignment $f\mapsto E_k(f)$ defines an immersion
$$
\psi_k:\,\mathscr{U}\to\Grass(b_{n,d}(k),S_k),
$$
Namely, $\psi_k$ is injective and its differential $\rd\psi_k$ is also injective at any point $f\in\mathscr{U}$.
\end{thm}

\begin{proof}
By the equivalence of (4) and (6) in Corollary \ref{cor: equi}, we have that $\psi_k$ is injective.

We will not distinguish an element $f\in \bb{P}(S_d)$ and its lifting in $S_d$. For $f\in\mathscr{U}$, we have $T_f\mathscr{U}=T_f\bb{P}(S_d)\simeq\Hom(\bb{C}f,S_d/\bb{C}f)$. The mapping $\Hom(\bb{C}f,S_d/\bb{C}f)\ni\eta\mapsto\eta(f)\in S_d/\bb{C}f$ then gives an identification $T_f\mathscr{U}\simeq S_d/\bb{C}f$. With the help of this identification, the differential of $\psi_k$ at $f\in\mathscr{U}$ is given by
$$
(\rd\psi_k)_f:\,T_f\mathscr{U}\simeq S_d/\bb{C}f\to\Hom(E_k(f),S_k/E_k(f)).
$$
Therefore, we have $(\rd\psi_k)_f(h)=0$ as an element of $\Hom(E_k(f),S_k/E_k(f))$ for any $h\in\Ker(\rd\psi_k)_f$. Represent $h$ by an element in $S_d/\bb{C}f$, and lift it to an element in $S_d$ which is still denoted by $h$. A direct computation gives that
$$
(\rd\psi_k)_f(h)(E_k(f))=E_k(h)+E_k(f) \mod E_k(f).
$$
Hence it follows from $(\rd\psi_k)_f(h)=0$ that $E_k(h)\subset E_k(f)$.

From the semicontinuity of the dimension of $E_k(f)$ with respect to $f\in S_d$, we obtain that for a small positive number $\epsilon>0$ and for any $t\in\bb{C}$ such that $|t|<\epsilon$, the following hold:

(i) $\dim E_k(f+th)=b_{n,d}(k)=\dim E_k(f)$;

(ii) $E_k(f+th)\subset E_k(f)$.

Hence $E_k(f+th)=E_k(f)$ for any $|t|<\epsilon$. In particular, choosing $t_0\neq 0$ satisfying $|t_0|<\epsilon$, we have $E_k(f+t_0h)=E_k(f)$. Using (4)$\Leftrightarrow$(6) in Corollary \ref{cor: equi} again, we deduce that $f+t_0h=f$ in $\bb{P}(S_d)$, hence $h=f$ in $\bb{P}(S_d)$ which implies that the chosen tangent vector $h\in\Ker(\rd\psi_k)_f$ is equal to zero. Therefore $(\rd\psi_k)_f$ is also injective.
\end{proof}

The above proof also gives the following corollary, which is interesting in its own right; compare with Corollary \ref{cor: equi}.

\begin{cor}
Given $f\in\mathscr{U}$ and $h\in\bb{P}(S_d)$. Suppose $E_k(h)\subset E_k(f)$ for some $d-1\leq k\leq T$, then $h=f$.
\end{cor}

\section{Polynomials of Sebastiani-Thom type}\label{secST}
In this section, we give a brief discussion about the fibers of the map $\varphi_k$ in \eqref{varphik} over $\varphi_k(f)$ for a polynomial $f$ of ST type.

By \cite[Proposition 4.8 or Corollary 3.15]{Fed2}, a smooth homogeneous polynomial $f\in S_d$ admits a unique maximally fine ``direct sum decomposition"
\begin{equation}\label{directsumdecomp}
 f(x_0,\cdots,x_n)=f_1(x_0,\cdots,x_{n_1-1})+f_2(x_{n_1},\cdots,x_{n_2})+\cdots+f_s(x_{n_{s-1}},\cdots,x_n),
\end{equation}
for a choice of linear coordinates $\{x_i\}_{i=0}^n$, where $0\leq n_1\leq n_2\leq\cdots\leq n_{s-1}\leq n$ and none of the $f_j$'s is of ST type. In addition, if $g\in S_d$ satisfies $E_{d-1}(g)\subset E_{d-1}(f)$, then necessarily, $g$ is of the following form
\begin{equation}\label{decompg}
g=\lambda_1f_1+\cdots+\lambda_sf_s,\quad\lambda_i\in\bb{C},
\end{equation}
see \cite[Corollary 3.12]{Fed2}. In particular, if $g$ is also smooth, then all the $\lambda_j$'s in \eqref{decompg} are nonzero. With these results at hand, we prove the following theorem.

\begin{thm}
For any $d-1\leq k\leq T=(n+1)(d-1)$ and any $f\in\bb{P}(S_d)_\Delta$, the fiber over $\varphi_k(f)$ of $\varphi_k$ defined in \eqref{varphik}, namely,
$$
\varphi_k:\,\bb{P}(S_d)_{\Delta}\to\Grass(S_k,a_{n,d}(k)),
$$
is
$$
\varphi_k^{-1}(\varphi_k(f))=\{\,\lambda_1f_1+\cdots+\lambda_sf_s\,|\,\lambda_i\in\bb{C}^*, i=1,\cdots,s\,\}.
$$
\end{thm}
\begin{proof}
It is obvious that for $\lambda_j$'s nonzero, the polynomial $\lambda_1f_1+\cdots+\lambda_sf_s$ is smooth and is mapped under $\varphi_k$ to $\varphi_k(f)$.

Conversely, if $g\in\bb{P}(S_d)_\Delta$ satisfies $\varphi_k(g)=\varphi_k(f)$, then we have $E_k(g)=E_k(f)$. It follows by (4)$\Rightarrow$(1) in Lemma \ref{Macaulayinv} that $E_{d-1}(g)=E_{d-1}(f)$. Hence by \cite[Corollary 3.12]{Fed2}, we have that $g$ is of the form $\lambda_1f_1+\cdots+\lambda_sf_s$ for nonzero $\lambda_j$'s.
\end{proof}

In conclusion, for the map $\varphi_k$, we can explicitly and completely determine all the fibers.


\bigskip
\begin{thebibliography}{00}

       \bibitem{AI1} J.~Alper and A.~Isaev, \emph{Associated forms and hypersurface singularities: The binary case}, to appear in J. reine angew. Math, published online, DOI: 10.1515/crelle-2016-0008.
	
     \bibitem{BBKT} W.~Buczy\'nska, J.~law Buczy\'nski, J.~Kleppe, Z.~Teitler, \emph{Apolarity and direct sum decomposability of polynomials}, Michigan Math. J. {\bf 64}(4), 675-–719, 2015.

        \bibitem{Dim87} A.~Dimca, \emph{Topics on real and complex singularities: an introduction}, Advanced Lectures in Mathematics, Vieweg, 1987.

	     \bibitem{Don} R.~Donagi, \emph{Generic Torelli for projective hypersurfaces}, Compos. Math. {\bf 50} (1983), 325–-353.

      \bibitem{Eis} D.~Eisenbud, \emph{Commutative algebra with a view toward algebraic geometry}, volume {\bf 150} of Graduate Texts in Mathematics, Springer-Verlag, New York, 1995.

         \bibitem{Fed1} M.~Fedorchuk, \emph{GIT semistability of Hilbert points of Milnor algebras}, Math. Ann {\bf 367}(2017), 441--460.

         \bibitem{Fed2} M.~Fedorchuk, \emph{Direct sum decompsoition of polynomials and factorization of associated forms}, arXiv:1705.03452v5.

      \bibitem{GH} T.~Gaffney, H.~Hauser, \emph{Characterizing singularities of varieties and of mappings}, Invent. math. {\bf 81} (1985), 427--447.

      \bibitem{GKZ} I.~M.~Gelfand, M.~M.~Kapranov, A.~V.~Zelevinsky, \emph{Discriminants, resultants, and multidimensional
determinants}, Mathematics: Theory \& Applications, Springer Science+Business Media, LLC, 1994.

      \bibitem{HMMNW} T.~Harima, T.~Maeno, H.~Morita, Y.~Numata, A.~Wachi, J.~Watanabe, \emph{The Lefschetz properties}, Lecture Notes in Mathematics {\bf 2080}, Springer, Heidelberg, 2013.

      \bibitem{HMNW} T.~Harima, J.~Migliore, U.~Nagel, J.~Watanabe, \emph{The weak and strong Lefschetz properties for artinian $K$-algebras}, Journal of Algebra {\bf 262}(2003), 99--126.

  \bibitem{HW} T.~Harima, J.~Watanabe, \emph{The strong Lefschetz property for Artinian algebras with non-standard grading}, J. Algebra {\bf 311} (2007), 511-–537.


   \bibitem{IK} A.~Iarrobino, V.~Kanev, \emph{Power sums, Gorenstein algebras, and determinantal loci}, volume {\bf 1721} of Lecture Notes in Mathematics, Springer-Verlag, Berlin, 1999.

   \bibitem{MN} J.~Migliore, U.~Nagel,, \emph{Survey article: a tour of the weak and strong lefschetz properties}, Journal of Commutative Algebra {\bf 5}(2013), 329--358.

    \bibitem{UY} K.~Ueda, M.~Yoshinaga, \emph{Logarithmic vector fields along smooth divisors in projective spaces}, Hokkaido Math. J. {\bf 38}(2009), 409-–415.

      \bibitem{ZW} Zhenjian Wang, \emph{On homogeneous polynomials determined by their Jacobian ideal}, Manuscripta Math. {\bf 146}(2015), 559--574.

      \bibitem{Voi} C.~Voisin, \emph{Hodge Theory and Complex Algebraic Geometry II}, Cambridge University Press, 2003.

\end{thebibliography}
\end{document}